\documentclass[11pt,a4paper,reqno]{article}
\usepackage{graphicx}
\usepackage{amsmath,amsthm,amsfonts,amssymb}
\usepackage[abbrev,non-sorted-cites]{amsrefs}
\usepackage{latexsym}
\usepackage{color}

\DeclareMathOperator{\Ric}{Ric}

\DeclareMathOperator{\Hess}{Hess}

\DeclareMathOperator{\im}{Im}
\DeclareMathOperator{\re}{Re}

\newtheorem{theorem}{Theorem}[section]
\theoremstyle{plain}

\newtheorem{corollary}[theorem]{Corollary}

\newtheorem{definition}[theorem]{Definition}
\newtheorem{example}{Example}

\newtheorem{lemma}[theorem]{Lemma}

\newtheorem{problem}{Problem}
\newtheorem{proposition}[theorem]{Proposition}
\newtheorem{remark}{Remark}

\begin{document}

\title{$f$-minimal Lagrangian Submanifolds in K\"ahler Manifolds with Real Holomorphy Potentials}
\author{WEI-BO SU}
\date{}
\maketitle
\begin{abstract}
The aim of this paper is to study variational properties for $f$-minimal Lagrangian submanifolds in K\"ahler manifolds with real holomorphy potentials. Examples of submanifolds of this kind incuding soliton solutions for Lagrangian mean curvature flow (LMCF). We derive second variation formula for $f$-minimal Lagrangians as a generalization of Chen and Oh's formula for minimal Lagrangians. As a corollary, we obtain stability of expanding and translating solitons for LMCF. We also define calibrated submanifolds with respect to $f$-volume in gradient steady K\"ahler--Ricci solitons as generalizations of special Lagrangians and translating solitons for LMCF, and show that these submanifolds are necessarily noncompact. As a special case, we study the exact deformation vector fields on Lagrangian translators. Finally we discuss some generalizations and related problems.
\end{abstract}

\section{Introduction}
\subsection{Motivation: Stability of Minimal Lagrangians}
The stability properties of minimal Lagrangian submanifolds in K\"ahler manifolds is studied by Chen {\cite{BYB}} and Oh {\cite{Oh}}. In particular, they derive a beautiful second variational formula as follows. Let $(X, J, \overline{\omega})$ be a K\"ahler manifold with metric $\overline{g}$, $F:L\looparrowright X$ be a minimal Lagrangian submanifold, and let $\{F_{t}\}$ be a smooth family of compactly supported normal deformations of $F$ with $F_{0} = F$ and $\frac{d}{dt}\big|_{t=0}F_{t} = \xi\in\Gamma_{c}(NL)$. Since $L$ is Lagrangian, $\xi$ is naturally identified with a $1$-form $\alpha_{\xi}:= F^{*}\iota_{\xi}\overline{\omega}\in\Gamma(T^{*}L)$. Then 
\begin{align}\label{CO}
(\delta^{2}V)_{F(L)}(\xi) := \frac{d^{2}}{dt^{2}}\big|_{t=0}V(F_{t}) = \int_{L}\left(|d\alpha_{\xi}|^{2}_{g} + |d^{*}_{g}\alpha_{\xi}|^{2} - \overline{\Ric}(\xi, \xi) \right)dV_{g},
\end{align}
where $g$ is the induced metric on $L$, and $\overline{\Ric}$ is the Ricci curvature of $\overline{g}$. A minimal Lagrangian $F:L\looparrowright X$ is called \textit{stable} if $(\delta^{2}V)_{F(L)}(\xi)\geq 0$ for all $\xi\in\Gamma_{c}(NL)$, \textit{Lagrangian stable} if $(\delta^{2}V)_{F(L)}(\xi)\geq 0$ for all $\xi\in\Gamma_{c}(NL)$ with $\alpha_{\xi}$ being closed, and \textit{Hamiltonian stable} if $(\delta^{2}V)_{F(L)}(\xi)\geq 0$ for all $\xi\in\Gamma_{c}(NL)$ with $\alpha_{\xi}$ being exact.
From $(\ref{CO})$ one can deduce that
\begin{description}
\item[(i)] if $\overline{\Ric}<0$, then any minimal Lagrangian in $X$ is strictly stable, 
\item[(ii)] if $(X, J,\:\overline{\omega})$ is positive K\"ahler-Einstein, that is, $\overline{\Ric} = c\cdot\overline{g}$ for some $c>0$, then a minimal Lagrangian in $X$ is Hamiltonian stable if and only if $\lambda_{1}(\Delta_{g})\geq c$, and
\item[(iii)] if $(X, J,\:\overline{\omega}, \Omega)$ is a Calabi--Yau manifold equipped with Ricci-flat K\"ahler metric, then any minimal Lagrangian submanifold in $X$ is stable, and Jacobi fields on $L$ are given by solutions of harmonic $1$-form equation
\begin{align*}
d\alpha = d^{*}_{g}\alpha = 0, \;\;\alpha\in\Gamma(T^{*}L).
\end{align*}
\end{description}
The fact (iii) can be obtained from a Calibrated Geometry point of view. It is known that any minimal Lagrangian submanifold in a Calabi--Yau manifold $(X, J,\:\overline{\omega}, \Omega)$ is {\it special Lagrangian}, that is, $F:L\looparrowright X$ is \textit{calibrated} by $\re(e^{-i\theta_{0}}\Omega)$, or equivalently, 
\begin{align}\label{SL}
F^{*}\omega = F^{*}\im(e^{-i\theta_{0}}\Omega) = 0
\end{align}
on $L$ for some phase $\theta_{0}\in S^{1}$. This guarantees that if $L$ is compact, $F:L\looparrowright X$ is not only minimal but volume-minimizing in its homology class. McLean {\cite{Mclean}} shows that the linearization of equation $(\ref{SL})$ is exactly the harmonic $1$-form equation,
thus any Jacobi field on $L$ also induces an infinitesimal special Lagrangian deformation. Moreover, the special Lagrangian deformations are unobstructed and the moduli space of special Lagrangians is a smooth manifold of dimension $b^{1}(L)$. 

\subsection{$f$-stability of $f$-minimal Lagrangians and LMCF Solitons}
In this paper, we aim to generalize the above story to Lagrangian submanifolds which are ``minimal'' with respect to certain weighted volume functionals. Consider a smooth function $f:X\to\mathbb{R}$ and the corresponding weighted volume form $e^{-f}dV_{\overline{g}}$ on a K\"ahler manifold $(X, J, \overline{\omega})$. An analogue quantity for Ricci curvature $\overline{\Ric}$ on the metric measure space $(X, \overline{g}, e^{-f}dV_{\overline{g}})$ is the symmetric $2$-tensor $\overline{\Ric}_{f} = \overline{\Ric} + \overline{\Hess} f$, called the Bakry--\'Emery Ricci tensor. Since $\overline{g}$ is K\"ahler, $\overline{\Ric}$ is Hermitian with respect to $J$. Hence it is natural to require $\overline{\Hess} f$ to be Hermitian. This is equivalent to requiring that $f$ to be a \textit{real holomorphy potential}, that is, $\overline{\nabla}^{1,0}f$ is a holomorphic vector field on $X$. 

Define the \textit{$f$-volume functional} on the space of $p$-submanifolds by
\begin{align*}
V_{f}: (F:L\looparrowright X)\longmapsto\int_{L}e^{-\frac{p}{2m}F^{*}f}\:dV_{g}\in\mathbb{R}_{+}\cup\{+\infty\}.
\end{align*}
Notice that $e^{-\frac{p}{2m}F^{*}f}dV_{g}$ is the volume of the induced conformal metric $F^{*}(e^{-\frac{f}{m}}\:\overline{g})$ on $L$. A $p$-submanifold $F:L\looparrowright X$ is a critical point with respect to $V_{f}$ if and only if the \textit{generalized mean curvature vector} $H + \frac{p}{2m}(\overline{\nabla}f)^{\perp} = 0$, where $H$ is the mean curvature vector of $F:L\looparrowright X$. Such a submanifold is called an \textit{$f$-minimal} submanifold. 

Under the above settings, we prove the following second variational formula for $f$-minimal Lagrangian submanifolds:
\begin{theorem}\label{Thm1}
Assume that $(X, J, \overline{\omega}, f)$ is a K\"ahler manifold with a real holomorphy potential, and $F: L\looparrowright X$ is an $f$-minimal Lagrangian submanifold. Then for any compactly supported normal variational vector field $\xi$ on $L$,
\begin{align}\label{SV}
(\delta^{2}V_{f})_{F(L)}(\xi) = \int_{L}\{\:|d\alpha_{\xi}|^{2}_{g} + |d^{*}_{f}\alpha_{\xi}|^{2} - \overline{\Ric}_{f}(\xi, \xi)\:\}\;e^{-\frac{1}{2}F^{*}f}dV_{g},
\end{align}
where $\alpha_{\xi} := F^{*}\iota_{\xi}\overline{\omega}$ is the $1$-form on $L$ associated to $\xi$, and $d^{*}_{f}$ is the adjoint of $d$ in the weighted space $L^{2}(\Lambda^{*}T^{*}L, e^{-\frac{1}{2}F^{*}f}dV_{g})$.
\end{theorem}
\noindent We call an $f$-minimal Lagrangian \textit{$f$-stable} if $(\delta^{2}V_{f})_{F(L)}(\xi)\geq 0$ for all compactly supported normal variational vector field $\xi$ on $L$, and the notions \textit{Lagrangian $f$-stable} and \textit{Hamiltonian $f$-stable} are defined analogously.

Typical examples of K\"ahler manifolds admitting real holomorphy potentials are \textit{gradient K\"ahler--Ricci solitons} (KR soliton for short), that is, K\"ahler manifolds satisfying
\begin{align*}
\overline{\Ric} + \overline{\Hess}f = c\cdot\overline{g}.
\end{align*}
When $f$ is constant, this equation reduces to K\"ahler-Einstein equation. In this sense, gradient KR solitons are generalizations of K\"ahler-Einstein manifolds. For $c=0$, the soliton is called \textit{steady}, for $c>0$, it is called \textit{shrinking}, and for $c<0$, it is called \textit{expanding}. We then obtain a corollary of Theorem $\ref{Thm1}$ similar to that obtained by Chen and Oh's formula $(\ref{CO})$, that every $f$-minimal Lagrangian in a steady or expanding gradient KR soliton is $f$-stable, and that a $f$-minimal Lagrangian in a shrinking gradient KR soliton is Hamiltonian $f$-stable if and only if $\lambda_{1}(\Delta_{f})\geq c$, where $\Delta_{f}$ is the Witten Laplacian on $L$ associated to $g$ and $F^{*}f$.

Our formula $(\ref{SV})$ can be applied to study the stability of soliton solutions for Lagrangian mean curvature flow (LMCF). In fact, if $X = \mathbb{C}^{m}$ with standard Euclidean metric $g_{0}$ and $f(z):=\pm\frac{|z|^{2}}{2}$, then $(\mathbb{C}^{m}, i, g_{0}, f)$ is a shrinking/expanding gradient KR soliton and the $f$-minimal Lagrangians are \textit{shrinking/expanding solitons for mean curvature flow}, respectively, and if $f(z):=\langle z, T\rangle$ for some fixed $T\in\mathbb{C}^{m}$, $(\mathbb{C}^{m}, i, g_{0}, f)$ becomes a gradient steady KR soliton and the $f$-minimal Lagrangians are \textit{translating solitons} (see Example 1 and 2). By Theorem $\ref{Thm1}$, we see that:
\begin{corollary}
Every expanding soliton and translating soliton for Lagrangian mean curvature flow is $f$-stable.
\end{corollary}
\noindent The stability of soliton solutions to mean curvature flow under certain weighted volume functional was first studied by Colding--Minicozzi {\cite{CM}} for shrinking solitons in hypersurface case, and generalized to higher codimensional case by Andrews--Li--Wei {\cite{ALW}}, Arezzo--Sun {\cite{AS}}, and Lee--Lue {\cite{LL}}, but note that their functional (``entropy'') is different from the $f$-volume functional. For stability of translating solitons, the second variation formula for translating hypersurfaces under $f$-volume was obtained by Xin in {\cite{Xin}}, Shahriyari {\cite{S}} studied the stability of graphical translating surfaces in $\mathbf{R}^{3}$, and Yang {\cite{Yang}} and Sun {\cite{Sun}} studied the Lagrangian translating solitons. In particular, Yang {\cite{Yang}} proved that every Lagrangian translating soliton is Hamiltonian $f$-stable, and Sun {\cite{Sun}} showed that they are actually Lagrangian $f$-stable. 

\subsection{$f$-special Lagrangians and Translating Solitons}
Next, we will focus on the case $c=0$, so we let $(X, J, \overline{\omega}, f)$ be a gradient steady KR soliton. In {\cite{RB}}, Bryant shows that there is a holomorphic volume form $\Omega_{f}$ such that $(X, J, \overline{\omega}, \Omega_{f})$ becomes an almost Calabi--Yau $m$-fold. The $m$-form $\re\Omega_{f}$ is a calibration on $X$ with respect to the conformal metric $e^{-\frac{f}{m}}\overline{g}$. A key observation is that, this fact can be rephrase as that $\re\Omega_{f}$ is a calibration with respect to the $f$-volume $e^{-f}dV_{\overline{g}}$ on $X$, and hence any submanifold calibrated by $\re\Omega_{f}$ minimize the $f$-volume. We call such submanifolds \textit{$f$-special Lagrangians ($f$-SLags)} and view them not only as generalizations of special Lagrangians in Calabi--Yau manifolds, but also generalizations of Lagrangian translating solitons in $\mathbb{C}^{m}$, since they evolved under LMCF by ``translation'' along the negative gradient vector field $-\frac{1}{2}\overline{\nabla}f$.

It turns out that every $f$-minimal Lagrangians in a gradient steady KR soliton $(X, J, \overline{\omega}, f)$ can be viewed as an $f$-SLag, and hence the $f$-stability also follows from this point of view. The $f$-SLag deformations can be characterized by the solutions of the $f$-harmonic $1$-form equation
\begin{align*}
d\alpha = d^{*}_{f}\alpha = 0,\quad\alpha\in\Gamma_{c}(T^{*}L).
\end{align*} 
Thus if $L$ is compact, the moduli space of $f$-SLags is a smooth manifold with dimension $b^{1}(L)$. But unfortunately we have a nonexistence result:
\begin{proposition}
There is no compact $f$-minimal Lagrangian in a gradient expanding or steady KR soliton $(X, J, \overline{\omega}, f)$.
\end{proposition}
\noindent Therefore to study the deformation theory of $f$-SLags, one needs to impose suitable asymptotic conditions. As an experiment, we study the case when $F:L\looparrowright\mathbb{C}^{m}$ is an exact Lagrangian translating soliton and assuming that the deformation is exact with weighted $L^{2}$ potential. We show that such deformation must be trivial on $L$, that is, there is no nonzero weighted $L^{2}$ $f$-harmonic function on $L$. 

\begin{proposition}
Suppose $F:L\looparrowright\mathbb{C}^{m}$ is a Lagrangian translating soliton and $u$ is a function on $L$ with $\|u\|_{L^{2}(e^{-\frac{1}{2}F^{*}f}dV_{g})} < \infty$ and $\Delta_{f}u = 0$. Then $u \equiv 0$.
\end{proposition}

To study the deformation theory of Lagrangian translating solitons further, one needs to impose more complicated asymptotic conditions and study the Fredholm theory of $\Delta_{f}$ in the corresponding weighted spaces. On the other hand, the properties of $f$-harmonic functions on noncompact $f$-minimal submanifolds might be useful for describing the topology at infinity. See {\cite{IR2}}, {\cite{IR1}} for some results in this direction in hypersurface case.

This paper is organized as follows. In Section $\ref{pre}$ we introduce the K\"ahler manifolds with real holomorphy potentials and $f$-minimal Lagrangian submanifolds. In Section $\ref{secSVF}$ we prove the second variation formula for the $f$-volume and stability of solitons for LMCF. We study $f$-calibrated submanifolds and prove a noncompact result in Section $\ref{secCali}$. In the final section, some generalizations and related problems are discussed.

\subsection*{Acknowledgements} Part of this paper were done when the author was visiting the Mathematics Institute in University of Oxford as a Recognised Student supervised by Professor Dominic Joyce, from January 2017 to June 2017. The author wants to express his gratitude to the Institute for the hospitality and to Professor Joyce for his kind advices and suggestions. He would also like to thank his advisor Professor Yng-Ing Lee for her constant encouragements and supports, and Professor Jason Lotay for the useful comments. The author is supported by MOST project 105-2115-M-002-004-MY3, and his visit to Oxford was also supported by National Center for Theoretical Sciences and Professor Dominic Joyce.

\section{Preliminaries}\label{pre}
\subsection{K\"ahler Manifolds with Real Holomorphy Potentials}
In the following, $(X, J)$ will be a smooth, connected, complex manifold with $\dim_{\mathbb{R}}X = 2m$, and $\overline{\omega}$ will be a K\"ahler form with K\"ahler metric $\overline{g}$. The Levi-Civita connection of $\overline{g}$ will be denoted by $\overline{\nabla}$, and the corresponding quantities with respect to $\overline{\nabla}$, such as Hessian and curvature, will be denoted by notations with overline.  

We will assume that there exists a function $f:X\to\mathbb{R}$ such that 
\begin{align}
\overline{\Hess}f(JX, JY) = \overline{\Hess}f(X, Y),
\end{align}
where $\overline{\Hess}$ is the Hessian of $f$ with respect to $\overline{g}$.
In fact, it is not hard to see that the following conditions are equivalent.
\begin{proposition}
Let $f:X\to\mathbb{R}$ be a smooth function. The following are equivalent:
\begin{enumerate}
\item[(i)] $\overline{\nabla}^{1,0}f$ is a holomorphic vector field,
\item[(ii)] $\overline{\Hess}f(JX, JY) = \overline{\Hess}f(X, Y)$ for any $X, Y\in\Gamma(TX)$,
\item[(iii)] $J\overline{\nabla}f$ is a Killing vector field.
\end{enumerate}
\end{proposition}
Such a function $f$ is called a \textit{real holomorphy potential} on $(X, J, \:\overline{\omega})$. Some properties of K\"ahler manifolds admitting real holomorphy potentials can be found in Munteanu--Wang {\cite{MW14}}, {\cite{MW15}}. Typical examples of manifolds of this kind are \textit{gradient K\"ahler--Ricci solitons}:

\begin{example}[Gradient K\"ahler--Ricci Solitons]\label{ExKRS}
{\rm Consider K\"ahler manifold $(X, J, \:\overline{\omega})$ together with a smooth function $f:X\to\mathbb{R}$ satisfying
\begin{align}\label{GKRS}
\overline{\Ric} + \overline{\Hess}f = c\cdot\overline{g},
\end{align}
for some $c\in\mathbb{R}$, where $\overline{\Ric}$ is the Ricci curvature of $\overline{g}$. Since
by K\"ahler condition we have $\overline{\Ric}(JX, JY) = \overline{\Ric}(X, Y)$, 
so $f$ must satisfy
\begin{align*}
\overline{\Hess}f(JX, JY) = \overline{\Hess}f(X, Y).
\end{align*}
Thus $f$ is a real holomorphy potential. The quadruple $(X, J, \overline{\omega}, f)$ is called a \textit{gradient K\"ahler--Ricci solitons} (\textit{KR solitons} for short). The gradient vector field $\overline{\nabla}f$ generates soliton solution to K\"ahler--Ricci flow (KRF) $\partial_{t}\overline{\omega}(t) = -\rho(\overline{\omega}(t))$, where $\rho(\overline{\omega}(t))$ is the Ricci form of $\overline{\omega}(t)$, in the following way. Define
\begin{align*}
\overline{g}(t) := \sigma(t)\varphi^{*}_{t}\overline{g},
\end{align*}
where $\sigma(t) = 1-ct$ and $\varphi_{t}$ is the flow on $X$ generated by $\frac{1}{2\sigma(t)}\overline{\nabla}f$. Then it is straightforward to verify that $\overline{\omega}(t)$ satisfies the KRF with $\overline{\omega}(0) = \overline{\omega}$ as long as $\sigma(t)>0$.

One can classify the gradient KR solitons into three classes in terms of the sign of the constant $c\in\mathbb{R}$:
\begin{description}
\item[(i)] When $c<0$, $(X, J, \overline{\omega}, f)$ is called a gradient \textit{expanding} soliton. The KRF $\overline{g}(t)$ evolves the metric $\overline{g}$ by homothetically expanding the length scale since $\sigma(t)' = -c>0$. For example, take $(X, J) = (\mathbb{C}^{m}, i)$ with Euclidean metric $\overline{\omega_{0}}$, and $f(z): = -\frac{|z|^{2}}{2}$, then $c = -1$. The resulting expanding soliton $(\mathbb{C}^{m}, i, \overline{\omega_{0}}, f)$ is called the expanding Gaussian soliton.

\item[(ii)] When $c>0$, $(X, J, \overline{\omega}, f)$ is called a gradient \textit{shrinking} soliton. The KRF $\overline{g}(t)$ evolves the metric $\overline{g}$ by homothetically shrinking the length scale since $\sigma(t)' = -c<0$. For example, take $(X, J) = (\mathbb{C}^{m}, i)$ with Euclidean metric $\overline{\omega_{0}}$, and $f(z): = \frac{|z|^{2}}{2}$, then $c = 1$. The resulting shrinking soliton $(\mathbb{C}^{m}, i, \overline{\omega_{0}}, f)$ is called the shrinking Gaussian soliton.

\item[(iii)] When $c = 0$, $(X, J, \overline{\omega}, f)$ is called a gradient \textit{steady} soliton. The KRF $\overline{g}(t)$ evolves the metric $\overline{g}$ by holomorphic reparametrizations $\varphi_{t}$. For example, take $(X, J) = (\mathbb{C}^{m}, i)$ with Euclidean metric $\overline{\omega_{0}}$, and $f(z): = \langle z, T\rangle$ with some fixed vector $T\in\mathbb{C}^{m}$, then  
$(\mathbb{C}^{m}, i, \overline{\omega_{0}}, f)$ is a steady soliton structure on $\mathbb{C}^{m}$.
\end{description}

Note that if $f$ is constant, condition $(\ref{GKRS})$ reduces to K\"ahler-Einstein condition. Hence gradient KR solitons can be viewed as generalizations of K\"ahler-Einstein manifolds.
}
\end{example}

\subsection{$f$-minimal Lagrangian Submanifolds}
Given a K\"ahler manifold with holomorphy potential $(X, J, \:\overline{\omega}, f)$.
Let $F:L\looparrowright X$ be an immersed, oriented, connected submanifold. The induced metric will be denoted by $g:= F^{*}\overline{g}$, and the corresponding quantities, such as Levi-Civita connection and curvature of $g$, will be denoted by notations without overline. Define the $f$-volume functional on the space of $p$-submanifolds by
\begin{align*}
(F:L\looparrowright X)\longmapsto V_{f}(F(L)):=\int_{L}e^{-\frac{p}{2m}F^{*}f}dV_{g}.
\end{align*}
Let $\{F_{t}:L\looparrowright X\}_{t\in(-\epsilon, \epsilon)}$ be a compactly supported normal variational of $F:L\looparrowright X$, with $F_{0} = F$ and $\frac{d}{dt}\big|_{t = 0}F_{t} = \xi\in\Gamma_{c}(NL)$, then the first variation formula of the $f$-volume is given by
\begin{align}\label{FV}
\frac{d}{dt}V_{f}(F_{t}(L))\big|_{t=0} = -\int_{L}\langle\: H + \frac{p}{2m}\overline{\nabla}f\:, \xi\:\rangle\;e^{-\frac{p}{2m}F^{*}f}dV_{g},
\end{align}
where $H$ is the mean curvature vector of $F:L\looparrowright X$, and $(\:\cdot\:)^{\perp}$ is projection to the normal bundle of $L$.
\begin{definition}
A $p$-submanifold $F: L\looparrowright X$ is called $f$-minimal if $H + \frac{p}{2m}(\overline{\nabla}f)^{\perp} = 0$.
\end{definition}
If $f$ is constant, $f$-minimal submanifolds are minimal submanifolds. Hence $f$-minimal submanifolds can be viewed as generalizations of minimal submanifolds.

In the following, we will consider only \textit{Lagrangian submanifolds}. Recall that
an $m$-submanifold $F:L^{m}\looparrowright X$ in a symplectic manifold $(X^{2m},\:\overline{\omega})$ is called \textit{Lagrangian} if $F^{*}\overline{\omega} = 0$. If $F:L\looparrowright X$ is Lagrangian, any compatible almost complex structure $J$ on $X$ gives rise to an isomorphism between normal bundle $NL$ and tangent bundle $TL$ of $L$. Then by composing with the induced metric $g$ we obtain an isomorphism between $NL$ and $T^{*}L$. We will use the same notation as in \cite{Oh} to denote such isomorphism.
\begin{definition}
Let $F:L\looparrowright X$ be a Lagrangian submanifold. Define an isomorphism $\widetilde{\omega}: NL\to T^{*}L$ by
\begin{align}\label{iso}
\widetilde{\omega}(\xi) := F^{*}(\iota_{\xi}\overline{\omega}),\;\;\;\xi\in\Gamma(NL).
\end{align}
\end{definition}
In a K\"ahler manifold, Dazord {\cite{Da}} shows that the mean curvature vector $H$ of a Lagrangian submanifold $F:L\looparrowright X$ satisfies $d\widetilde{\omega}(H) = F^{*}\overline{\rho}$, where $\overline{\rho} = \overline{\Ric}(J\cdot, \cdot)$ is the Ricci form. Thus if $(X, J, \overline{\omega})$ is K\"ahler-Einstein, then $\widetilde{\omega}(H)$ is closed, so it induces an infinitesimal  Lagrangian deformation. Furthermore, Smoczyk {\cite{Smo}} shows that the Lagrangian condition is preserved by the mean curvature flow $\frac{d}{dt}F_{t} = H(F_{t})$ whenever $(X, J, \overline{\omega})$ is K\"ahler-Einstein. Therefore, it is reasonable to consider \textit{Lagrangian mean curvature flow} (LMCF) in K\"ahler-Einstein manifolds and soliton solutions for LMCF.

The next example explains the meaning of LMCF solitons and shows that they can be viewed as $f$-minimal Lagrangian submanifolds in $\mathbb{C}^{m}$.
\begin{example}[Soliton Solutions for LMCF]
{\rm Consider $(X, J) = (\mathbb{C}^{m}, i)$, $\overline{\omega} = $ Euclidean metric $\overline{\omega_{0}}$. 
\begin{description}
\item[(i)] Define $f(z):= \frac{|z|^{2}}{2}$, then $f$ is a real holomorphy potential and $\overline{\nabla}f(z) = z$. Then any $f$-minimal Lagrangian submanifold $F:L\looparrowright\mathbb{C}^{m}$ satisfies
\begin{align}\label{shrinker}
H + \frac{1}{2}F^{\perp} = 0.
\end{align}
Lagrangian submanifolds in $\mathbb{C}^{m}$ satisfying $(\ref{shrinker})$ are called \textit{shrinking soliton} for LMCF. Indeed, the homothetically shrinking family about the origin $\{F_{t} = \sqrt{1-t}F\}_{t<1}$ satisfies LMCF with $F_{0} = F$.
\item[(ii)] Similarly, the $f$-minimal Lagrangian submanifolds in $\mathbb{C}^{m}$ with $f(z):=-\frac{|z|^{2}}{2}$ satisfies
\begin{align}\label{expander}
H - \frac{1}{2}F^{\perp} = 0.
\end{align}
Lagrangian submanifolds in $\mathbb{C}^{m}$ satisfying $(\ref{expander})$ are called \textit{expanding soliton} for LMCF. Indeed, the homothetically expanding family about the origin $\{F_{t} = \sqrt{1+t}F\}_{t>-1}$ satisfies LMCF with $F_{0} = F$.
\item[(iii)] Let $T\in\mathbb{C}^{m}$ be a fixed vector, define $f(z):= 2\langle z, T\rangle$. Then $\overline{\nabla}f = 2T$. The $f$-minimal Lagrangian submanifolds in $\mathbb{C}^{m}$ satisfies
\begin{align}\label{translator}
H + T^{\perp} = 0.
\end{align}
Lagrangian submanifolds in $\mathbb{C}^{m}$ satisfying $(\ref{translator})$ are called \textit{translating soliton} for LMCF. Indeed, the family moving by translation in $T$-direction $\{F_{t} = F - tT\}_{t\in\mathbb{R}}$ satisfies LMCF with $F_{0} = F$.
\end{description}
}
\end{example}

\section{Second Variation Formula and Stability of $f$-minimal Lagrangian Submanifolds}\label{secSVF}
First we introduce the differential operators that will be used in the following sections. 
Let $(X, J, \overline{\omega}, f)$ be a K\"ahler manifold with a real holomorphic potential and $F:L\looparrowright X$ be a Lagrangian submanifold with induced metric $g$. To simplify the notations, we will continue to use $f$ to denote the restriction of the ambient function $f$ to $L$. Consider the space of weighted $L^{2}$ differential forms $L^{2}(\Lambda^{*}T^{*}L, e^{-\frac{f}{2}}dV_{g})$ on $L$ with inner product
\begin{align*}
(\:\cdot\:,\:\cdot\:)_{f} := \int_{L}\langle\:\cdot\:,\:\cdot\:\rangle_{g}\:e^{-\frac{f}{2}}dV_{g}.
\end{align*}
Then the formal adjoint of $d$ with respect to $(\cdot,\cdot)_{f}$ is given by $d^{*}_{f} := d^{*}_{g} + \frac{1}{2}\iota_{\nabla f}$. Define
\begin{align*}
\Delta_{f} : = d\:d^{*}_{f} + d^{*}_{f}\:d,
\end{align*}
then $\Delta_{f}$ is a positive-definite self-adjoint operator with respect to $(\cdot, \cdot)_{f}$. This operator is usually called the \textit{Witten Laplacian} associated to $f$.

We are now ready to prove the second variation formula for $V_{f}$.
\begin{theorem}\label{SVF}
Assume that $(X, J, \overline{\omega}, f)$ is a K\"ahler manifold with a real holomorphy potential, and $F: L\looparrowright X$ is an $f$-minimal Lagrangian submanifold. Then for any compactly supported normal variation $\{F_{t}\}_{t\in(-\epsilon, \epsilon)}$ with $F_{0} = F$ and $\frac{d}{dt}\big|_{t=0}F_{t} = \xi\in\Gamma_{c}(NL)$, we have
\begin{align}
(\delta^{2}V_{f})_{F(L)}(\xi) = \int_{L}\{\:\langle\:\widetilde{\omega}^{-1}\circ\Delta_{f}\circ\widetilde{\omega}(\xi),\: \xi\:\rangle - [\:\overline{\Ric}(\xi, \xi) + \overline{\Hess}f(\xi,\xi)\:]\:\}\;e^{-\frac{f}{2}}dV_{g}.
\end{align}
\end{theorem}

\begin{proof}
Differentiate ($\ref{FV}$) again and use the $f$-minimal condition $H + \frac{1}{2}(\overline{\nabla}f)^{\perp} = 0$,
\begin{align*}
\frac{d^{2}}{dt^{2}}V_{f}(L_{t})\big|_{t=0} = -\int_{L}[\:\langle\:\overline{\nabla}_{\xi}(H + \frac{1}{2}\overline{\nabla}f), \:\xi
\:\rangle + \langle\:\frac{1}{2}\nabla f, \overline{\nabla}_{\xi}\xi\:\rangle\:]\;e^{-\frac{f}{2}}dV_{g}.
\end{align*}
By the same computation as in the second variation formula for minimal submanifolds,
\begin{align*}
\langle\:\overline{\nabla}_{\xi}H, \:\xi\:\rangle = \langle\:\Delta^{N}\xi -\mathcal{R}(\xi) + \tilde{A}(\xi), \:\xi\:\rangle,
\end{align*}
where $\langle\:\mathcal{R}(\xi), \:\xi\:\rangle = \sum_{i=1}^{m}\langle\:\overline{R}(e_{i}, \xi)e_{i}, \:\xi\:\rangle$ for any orthonormal basis $\{e_{i}\}$ on $L$ and $\tilde{A} = A^{t}A$ for $A$ denoting the second fundamental form. We also compute
\begin{align*}
\langle\:\overline{\nabla}_{\xi}\overline{\nabla}f, \:\xi\:\rangle = \overline{\Hess}f(\xi, \xi),
\end{align*}
and
\begin{align*}
\langle\:\nabla f, \overline{\nabla}_{\xi}\xi\:\rangle = -\langle\:\overline{\nabla}_{\xi}\nabla f, \:\xi \:\rangle = -\langle\:\overline{\nabla}_{\nabla f}\xi, \:\xi\:\rangle,
\end{align*}
where in the last equality we use the fact that $\langle\:[\xi, \nabla f], \:\xi\:\rangle = 0$.
Combining these three terms we get
\begin{align*}
\frac{d^{2}}{dt^{2}}V_{f}(L_{t})\big|_{t=0} = \int_{L}[\:\langle\:(-\Delta^{N}\xi +\mathcal{R}(\xi) - \tilde{A}(\xi)) +  \frac{1}{2}\overline{\nabla}_{\nabla f}\xi, \:\xi\:\rangle - \frac{1}{2}\overline{\Hess}f(\xi, \xi)\:]\;e^{-\frac{f}{2}}dV_{g}.
\end{align*}
By Gauss formula and $f$-minimality, given any orthonormal basis $\{e_{i}\}$ on $L$,
\begin{align*}
&\langle\:\mathcal{R}(\xi), \:\xi\:\rangle = -\sum_{i=1}^{m}\langle\:\overline{R}(e_{i}, \xi)\xi, \:e_{i}\:\rangle\\
 &= -\overline{\Ric}(\xi,\xi) + \sum_{i=1}^{m}\langle\:\overline{R}(Je_{i}, \xi)\xi, \:Je_{i}\:\rangle\\
 &= -\overline{\Ric}(\xi,\xi) + \sum_{i=1}^{m}\langle\:\overline{R}(e_{i}, J\xi)J\xi, \:e_{i}\:\rangle\\
 &= -\overline{\Ric}(\xi,\xi) + \sum_{i=1}^{m}\langle\: R(e_{i}, J\xi)J\xi, \:e_{i}\:\rangle + \langle\: B(e_{i}, J\xi), B(e_{i}, J\xi)\:\rangle- \langle\: B(e_{i}, e_{i}), B(J\xi, J\xi)\:\rangle\\
 &= -\overline{\Ric}(\xi,\xi) + \Ric(J\xi, J\xi) + \langle\:\tilde{A}\xi, \:\xi\:\rangle - \frac{1}{2}\langle\:\overline{\nabla}_{J\xi}(\overline{\nabla}f)^{\perp}, \:J\xi\:\rangle. 
\end{align*}

On the other hand, we have the following relation proved by Oh ({\cite{Oh}}, Lem.3.3):
\begin{lemma}\label{Lem}
For any $\xi\in\Gamma(NL)$ we have
\begin{enumerate}
\item[(i)] $\nabla_{X}(\widetilde{\omega}(\xi)) = \widetilde{\omega}((\overline{\nabla}_{X}\xi)^{\perp})$ for all $X\in\Gamma(TL)$, and
\item[(ii)] $\Delta^{N}\xi = \widetilde{\omega}^{-1}\circ\Delta\circ\widetilde{\omega}(\xi)$, where $\Delta$ is the covariant Laplacian acting on $\Gamma(T^{*}L)$.
\end{enumerate}
\end{lemma}
\noindent By Lemma $\ref{Lem}$ (ii) and Weizenb\"ock formula we have
\begin{align*}
\Delta^{N}\xi = -\widetilde{\omega}^{-1}\circ(\Delta_{h} - \Ric)\circ\widetilde{\omega}(\xi),
\end{align*}
where $\Delta_{h} = dd^{*}_{g} + d^{*}_{g}d $ is the Hodge Laplacian.
Now the $f$-Laplacian acting on $1$-forms on $L$ is given by
\begin{align*}
\Delta_{f} =  d\:d^{*}_{f} + d^{*}_{f}\:d = \Delta_{h} + \frac{1}{2}\mathcal{L}_{\nabla f},
\end{align*} 
so expressing the Lie derivative by covariant derivative we get
\begin{align*}
\langle\:\widetilde{\omega}^{-1}\circ(\Delta_{h} - \Ric)\circ\widetilde{\omega}(\xi), \:\xi\:\rangle &= \langle\:\widetilde{\omega}^{-1}\circ(\Delta_{f} - \frac{1}{2}\mathcal{L}_{\nabla f} - \Ric)\circ\widetilde{\omega}(\xi),\:\xi\:\rangle\\
&= \langle\:\widetilde{\omega}^{-1}\circ(\Delta_{f} - \frac{1}{2}\nabla_{\nabla f})\circ\widetilde{\omega}(\xi), \:\xi\:\rangle - \frac{1}{2}\langle\:\nabla_{J\xi}\nabla f, \:J\xi\:\rangle - \Ric(J\xi, J\xi)\\
&= \langle\:\widetilde{\omega}^{-1}\circ\Delta_{f}\circ\widetilde{\omega}(\xi) - \frac{1}{2}\overline\nabla_{\nabla f}\xi, \:\xi\:\rangle - \frac{1}{2}\langle\:\nabla_{J\xi}\nabla f, \:J\xi\:\rangle - \Ric(J\xi, J\xi),
\end{align*}
where in the third line we use the Lemma $\ref{Lem}$ (i) to show that $\widetilde{\omega}^{-1}\circ\nabla_{\nabla f}\circ\widetilde{\omega} = (\overline{\nabla}_{\nabla f}\xi)^{\perp}$.

Combining everything together we finally obtain
\begin{align*}
\frac{d^{2}}{dt^{2}}V_{f}(L_{t})\big|_{t=0} = \int_{L}[\:\langle\: \widetilde{\omega}^{-1}\circ\Delta_{f}\circ\widetilde{\omega}(\xi), \:\xi\:\rangle - \overline{\Ric}(\xi, \xi) - \frac{1}{2}\overline{\Hess}f(\xi, \xi) - \frac{1}{2}\overline{\Hess}f(J\xi, J\xi)\:]\;e^{-\frac{f}{2}}dV_{g}.
\end{align*}
Notice that, by the assumption on $f$, 
\begin{align*}
\frac{1}{2}\overline{\Hess}f(\xi, \xi) + \frac{1}{2}\overline{\Hess}f(J\xi, J\xi) = \overline{\Hess}f (\xi, \xi).
\end{align*}
\end{proof}
Now we can define the notions of $f$-stability for $f$-minimal Lagrangians.
\begin{definition}
Let $(X, J, \overline{\omega}, f)$ be a K\"ahler manifold with a real holomorphy potential. An $f$-minimal Lagrangian $F:L\looparrowright X$ is called
\begin{enumerate}
\item[(i)] \textbf{f-stable} if $(\delta^{2}V_{f})_{F(L)}(\xi)\geq 0$ for all $\xi\in\Gamma_{c}(NL)$,
\item[(ii)] \textbf{Lagrangian f-stable} if $(\delta^{2}V_{f})_{F(L)}(\xi)\geq 0$ for $\xi\in\Gamma_{c}(NL)$ with $\widetilde{\omega}(\xi)$ being closed, and
\item[(iii)] \textbf{Hamiltonian f-stable} if $(\delta^{2}V_{f})_{F(L)}(\xi)\geq 0$ for $\xi\in\Gamma_{c}(NL)$ with $\widetilde{\omega}(\xi)$ being exact.
\end{enumerate}
\end{definition}
\noindent By the same proof as in {\cite{Oh}} Theorem 3.6 and Theorem 4.4, we have
\begin{corollary}
Let $(X, J, \overline{\omega}, f)$ be a K\"ahler manifold with a real holomorphy potential with Bakry--\'Emery Ricci curvature $\overline{\Ric}_{f} := \overline{\Ric} + \overline{\Hess}f$.
\begin{enumerate}
\item[(i)] If $\overline{\Ric}_{f} \leq 0$, then any $f$-minimal Lagrangian is $f$-stable.
\item[(ii)] If $\overline{\Ric}_{f} = c\cdot\overline{g}$ with $c>0$, then any $f$-minimal Lagrangian is Hamiltonian $f$-stable if and only if $\lambda_{1}(\Delta_{f})\geq c$.
\end{enumerate}
\end{corollary}
\noindent Notice that if we take $f$ to be a constant, this corollary reduces to Oh's original results. 

From Example 2 and Theorem $\ref{SVF}$, we obtain the $f$-stability for LMCF solitons.
\begin{corollary}
\begin{enumerate}
\item[(i)] Every Lagrangian expanding soliton is strictly $f$-stable.
\item[(ii)] Every Lagrangian translating soliton is $f$-stable.
\end{enumerate}
\end{corollary}

\begin{remark}\rm{
It is known that shrinking solitons for MCF are $f$-unstable, so one has to consider stability with respect to the ``entropy'' defined by Coding--Minicozzi {\cite{CM}}, called the \textit{$F$-stability}. See {\cite{LZ}} for some $F$-stability criterions for closed Lagrangian shrinking solitons.
}
\end{remark}

\section{Calibrated Submanifolds with respect to the $f$-volume}\label{secCali}
\subsection{$f$-special Lagrangian Submanifolds}
Recall that Harvey and Lawson \cite{HL} shows that if $(X, J, \overline{\omega}, \Omega)$ is a Calabi--Yau $m$-fold, then for any Lagrangian submanifold $F:L\looparrowright X$ we have
\begin{align*}
F^{*}\Omega = e^{i\theta}dV_{g}
\end{align*}
for some $\theta:L\to\mathbb{R}/2\pi\mathbb{Z}$, called the \textit{Lagrangian angle}. The mean curvature vector satisfies $H = J\nabla\theta$, thus if $F:L\looparrowright X$ is minimal, then $\theta = \theta_{0}$ is a constant. Moreover, in this case $F:L\looparrowright X$ is calibrated by $\re(e^{-i\theta_{0}}\Omega)$ and hence it is actually volume-minimizing in its homology class.

We now generalize the above theory to Lagrangian submanifolds in gradient steady KR solitons, and give an alternative description of $f$-stability for $f$-minimal Lagrangians. Given a gradient steady KR soliton $(X, J,\:\overline{\omega}, f)$, Robert Bryant {\cite{RB}} shows that there exists a nonvanishing holomorphic volume form, denoted by $\Omega_{f}$, such that
\begin{align}\label{ACY}
e^{-f}\:\frac{\overline{\omega}^{\:m}}{m!} = (-1)^{\frac{m(m-1)}{2}}\left(\frac{i}{2}\right)^{m}\Omega_{f}\wedge\overline{\Omega_{f}}.
\end{align}
In other words, $(X, J, \:\overline{\omega}, \:\Omega_{f})$ is \textit{almost Calabi--Yau} in the sense of Joyce (see {\cite{Jb2}, Def. 8.4.3}).
Define $\widetilde{g}: = e^{-\frac{f}{m}}\overline{g}$, then for any $\theta_{0}\in\mathbb{R}/2\pi\mathbb{Z}$, $\re{(e^{-i\theta_{0}}\Omega_{f})}$ is a calibration with respect to the conformal metric $\widetilde{g}$. We rephrase this from the view point of the $f$-volume.
\begin{definition}
\begin{enumerate}
\item[(i)] A $p$-form $\alpha$ on $(X, \overline{g}, f)$ is called an $f$-calibration if $d\alpha = 0$ and $\alpha\big|_{P}\leq e^{-\frac{p}{2m}f}\:vol_{P}$ for any $p$-dimensional oriented subspace $P\subset T_{x}X$, for all $x\in X$. 
\item[(ii)] A $p$-submanifold $F:L\looparrowright X$ in $X$ is said to be $f$-calibrated by an $f$-calibration $\alpha$ if 
\begin{align*}
F^{*}\alpha = e^{-\frac{p}{2m}F^{*}f}dV_{g},
\end{align*}
where $dV_{g}$ is the induced volume on $L$.
\end{enumerate}
\end{definition}
It is not hard to see that any $f$-calibrated submanifold is $f$-minimal and any compact $f$-calibrated submanifold minimizes the $f$-volume in its homology class. One can show that $\re{(e^{-i\theta_{0}}\Omega_{f})}$ is an $f$-calibration and the $f$-calibrated submanifolds are $f$-minimal Lagrangian submanifolds. Conversely, by choosing an orientation, any $f$-minimal Lagrangian submanifolds in a gradient steady KR soliton is $f$-calibrated by $\re{(e^{-i\theta_{0}}\Omega_{f})}$ for some $\theta_{0}\in\mathbb{R}/2\pi\mathbb{Z}$. 
\begin{definition}
We call the submanifolds calibrated by $\re{(e^{-i\theta_{0}}\Omega_{f})}$ \textit{f-special Lagrangian} submanifolds ($f$-SLag for short) with phase $\theta_{0}$.
\end{definition}
If $F$ is Lagrangian, then by the same method as in {\cite{HL}} one can show that $F^{*}\Omega_{f} = e^{i\theta}e^{-\frac{F^{*}f}{2}}dV_{g}$ for some $\theta\in\mathbb{R}/2\pi\mathbb{Z}$. We still call $\theta$ the \textit{Lagrangian angle}. It turns out that $F: L\looparrowright X$ is an $f$-SLag with phase $\theta_{0}$ if and only if $F: L\looparrowright X$ is Lagrangian with constant Lagrangian angle $\theta_{0}$
\begin{remark}
\rm{In Joyce's terminology (see {\cite{Jb3}, Def. 8.4.4}), $f$-SLags in our sense are still called \textit{special Lagrangians}. We put an $f$ here to emphasize the role of the real holomorphy potential and the relation to $f$-minimal submanifolds.
}
\end{remark}

We now give a family of examples of $f$-SLags in $\mathbb{C}^{m}$.
\begin{example}[Lagrangian Translating Solitons]\label{ExTrans}
{\rm Consider a Lagrangian translating soliton $F: L\looparrowright\mathbb{C}^{m}$ with Lagrangian angle $\theta$. Then as in {\cite{NT}},
\begin{align*}
\begin{array}{rcl}
H + T^{\perp} = 0 &\Longleftrightarrow & J\nabla\theta + T^{\perp} = 0\\
&\Longleftrightarrow & \nabla\theta - (JT)^{T} = 0\\
&\Longleftrightarrow & \nabla\theta - \nabla\langle F, JT\rangle = 0.
\end{array}
\end{align*}
So $F$ satisfies the \textit{translator equation} 
\begin{align}\label{TLS}
\theta(F) - \langle F, JT\rangle = \theta_{0}
\end{align}
for some constant $\theta_{0}\in\mathbb{R}$. We shall show that Lagrangian translating solitons are $f$-calibrated with phase $\theta_{0}$.

Let $\Omega:=dz^{1}\wedge\cdots\wedge dz^{m}$ and  
\begin{align*}
\Omega_{f}:=e^{-\frac{f}{2} - i\langle z, JT \rangle}\Omega,
\end{align*}
where $f(z) = 2\langle z, T\rangle$. Then $\Omega_{f}$ is a holomorphic volume form on $\mathbb{C}^{m}$ and satisfies $(\ref{ACY})$. Hence $(\mathbb{C}^{m}, i, \:\overline{\omega_{0}}, \:\Omega_{f})$ is almost Calabi--Yau. By Lagrangian condition we then have
\begin{align*}
F^{*}(\Omega_{f}) = e^{i(\theta(F) - \langle F, JT\rangle)}e^{-\frac{F^{*}f}{2}}dV_{g} = e^{i\theta_{0}}(e^{-\frac{F^{*}f}{2}}dV_{g}).
\end{align*}
Therefore Lagrangian translating solitons are $f$-SLag with phase $\theta_{0}$.\qed
}
\end{example}
When $f = 0$, $f$-SLags reduce to SLags in Calabi--Yau manifolds. Hence $f$-minimal Lagrangian submanifolds in gradient steady KR solitons are generalizations of special Lagrangians in Calabi--Yau manifolds. From Example $\ref{ExTrans}$, they can also be considered as generalizations of Lagrangian translating solitons in $\mathbb{C}^{m}$. In fact, under MCF, $f$-SLags are evolved by ``translation'' along the flow of the vector field $-\frac{1}{2}\overline{\nabla}f$ on $X$ (see section $\ref{KRMCF}$).

We would like to study the deformation theory of $f$-SLags. The equation for $f$-SLag deformation vector fields is given by the next lemma.
\begin{lemma}\label{fharlemma}
Let $(X, J, \:\overline{\omega}, f)$ be a steady KR soliton and $F:L\rightarrow X$ be an $f$-SLag. Then the $f$-SLag deformation vector fields $\xi\in\Gamma(NL)$ is characterized by the $f$-harmonic $1$-form equation
\begin{align}\label{fhar}
d\:\widetilde{\omega}(\xi) = d^{*}_{f}\:\widetilde{\omega}(\xi) = 0.
\end{align} 
\end{lemma}
\begin{proof}
Without loss of generality, we may assume $F$ has phase $0$. Let $\{F_{t}\}$ be a family of immersions satisfies $F_{0} = F$ and $\frac{d}{dt}\big|_{t=0}F_{t} = \xi\in\Gamma(NL)$. Then $\xi$ preserves $f$-SLag condition if and only if
\begin{align*}
\frac{d}{dt}\big|_{t=0}F^{*}_{t}\overline{\omega} = \frac{d}{dt}\big|_{t=0}F^{*}_{t}\im(\Omega_{f}) = 0.
\end{align*}
It is well known that $\frac{d}{dt}\big|_{t=0}F^{*}_{t}\overline{\omega} = 0$ if and only if $d\:\widetilde{\omega}(\xi) = 0$. Then we compute
\begin{align*}
\frac{d}{dt}\big|_{t=0}F^{*}_{t}\Omega_{f} = &F^{*}\mathcal{L}_{\xi}\Omega_{f} = F^{*}(d\iota_{\xi}\Omega_{f})\\
=& -i\:d\:\iota_{J\xi}F^{*}\Omega_{f}\\
=& -i\:d(\iota_{J\xi}e^{-\frac{1}{2}F^{*}f}dV_{g})\\
=& -i\:(-\frac{1}{2}d(F^{*}f)e^{-\frac{1}{2}F^{*}f}\wedge\iota_{J\xi}dV_{g} + e^{-\frac{1}{2}F^{*}f}\:d\:\iota_{J\xi}dV_{g})\\
=& \:i\:(d^{*}_{g}\:\widetilde{\omega}(\xi) + \frac{1}{2}\iota_{\nabla F^{*}f}\widetilde{\omega}(\xi))\:e^{-\frac{1}{2}F^{*}f}dV_{g}.
\end{align*}
Therefore $\frac{d}{dt}\big|_{t=0}F^{*}_{t}\im(\Omega_{f}) = 0$ if and only if
\begin{align*}
d^{*}_{g}\:\widetilde{\omega}(\xi) + \frac{1}{2}\iota_{\nabla F^{*}f}\widetilde{\omega}(\xi) = d^{*}_{f}\widetilde{\omega}(\xi) = 0.
\end{align*}
\end{proof}
Notice that $(\ref{fhar})$ also appears in the second variation formula for $f$-volume (Theorem $\ref{SVF}$) as the Jacobi field equation for $f$-minimal Lagrangians.

From Lemma $\ref{fharlemma}$, for \textit{compact} $f$-SLags, the deformation theory is the same as special Lagrangians, as shown by the following theorem:
\begin{theorem}[{\cite{Jb3}}, Thm.10.8]
Let $(X, J, \:\overline{\omega}, \:\Omega_{f})$ be an almost Calabi--Yau $m$-fold and $F:L^{m}\looparrowright X$ be a compact $f$-SLag submanifold. Then the moduli space of $f$-SLags is a smooth manifold of dimension $b^{1}(L)$.
\end{theorem}

Unfortunately, just like minimal submanifolds in Euclidean space, we have the following noncompact result for $f$-minimal Lagrangians.
\begin{proposition}
Let $(X, J,\:\overline{\omega}, f)$ be a gradient steady or expanding K\"ahler--Ricci soliton which is not K\"ahler-Einstein, and $F:L\looparrowright X$ be an $f$-minimal Lagrangian submanifold. Then $L$ must be noncompact.
\end{proposition}
\begin{proof}
Let $X,\:Y$ be tangent vectors of $L$. Then by $f$-minimality, 
\begin{align*}
\overline{\Hess}f(X, Y) &= \langle\overline{\nabla}_{X}\overline{\nabla}f, Y\rangle\\
 &= \langle\overline{\nabla}_{X}\nabla f, Y\rangle + \langle\overline{\nabla}_{X}(\overline{\nabla}f)^{\perp}, Y\rangle\\
 &= \Hess f(X, Y) - 2\langle\overline{\nabla}_{X}H, Y\rangle\qquad[\:\mbox{Since $H + \frac{1}{2}(\overline{\nabla}f)^{\perp} = 0$}.]\\
 &= \Hess f(X, Y) + 2A(H, X, Y).
\end{align*}

We first deal with the steady case. Since $\overline{\Ric} + \overline{\Hess}f = 0$,
\begin{align*}
\overline{\Ric}(X, Y) = -\Hess f(X, Y) -2A(H, X, Y).
\end{align*}
Let $\{e_{i}\}_{i=1}^{m}$ be an orthonormal basis of $T_{p}L$ for some $p\in L$. Then
\begin{align*}
\sum_{i=1}^{m}\overline{\Ric}(e_{i}, e_{i}) &= -\Delta f - 2|H|^{2}\\
 &= -\Delta f + \frac{1}{2}|\nabla f|^{2} - \frac{1}{2}|\overline{\nabla}f|^{2}\qquad[\:\mbox{Since $\overline{\nabla}f = \nabla f - 2H$}.]
\end{align*}
Since $F:L\looparrowright X$ is Lagrangian, $\{e_{i}, Je_{i}\}_{i=1}^{m}$ is an orthonormal basis of $T_{p}X$, Hence we have
\begin{align}\label{Ricci}
\overline{R} = \sum_{i=1}^{m}\overline{\Ric}(e_{i}, e_{i}) + \sum_{i=1}^{m}\overline{\Ric}(Je_{i}, Je_{i}) = 2\sum_{i=1}^{m}\overline{\Ric}(e_{i}, e_{i}).
\end{align}
Combining these equations we obtain
\begin{align*}
\overline{R} + |\overline{\nabla}f|^{2} = -2\Delta f + |\nabla f|^{2}.
\end{align*}
Now by Cao-Hamilton {\cite{CH}}, the quantity $\overline{R} + |\overline{\nabla}f|^{2}$ is constant on $X$ (see also {\cite{RB}} for a different proof). Therefore $f$ satisfies
\begin{align*}
\Delta f - \frac{1}{2}|\nabla f|^{2} = c
\end{align*}
for some $c\in\mathbb{R}$ on $L$. The result then follows from maximum principle.

Next we prove the expanding case. We may assume $\overline{\Ric} + \overline{\Hess}f = -\overline{g}$. For any vectors $X, Y$ tangent to $L$ we have
\begin{align*}
\overline{\Ric}(X, Y) + \Hess f(X, Y) + 2A(H, X, Y) = -\overline{g}(X, Y).
\end{align*}
Taking the tangential trace on both sides and use $(\ref{Ricci})$ we get
\begin{align}\label{noncpt}
\Delta f -\frac{1}{2}|\nabla f|^{2} + \frac{1}{2}|\overline{\nabla}f|^{2} + \frac{\overline{R}}{2} = - m.
\end{align}
On any gradient expanding soliton we know that (see {\cite{MC}}), after adding a suitable constant to $f$, $\overline{R} + |\overline{\nabla}f|^{2} + 2f = 0$. Hence $f$ satisfies
\begin{align*}
\Delta f -\frac{1}{2}|\nabla f|^{2} -f =  -m.
\end{align*}
Assume $x_{0}\in L$ is a local minimum of $f$ restricted on $L$, then
\begin{align*}
0\leq\Delta f (x_{0}) = -m + f(x_{0})\;\Longrightarrow\;f\big|_{L}\geq m.
\end{align*}
But from {\cite{SZ}} Corollary 2.4, the scalar curvature is bounded from below $\overline{R}\geq -2m$, so
\begin{align*}
2f = -\overline{R}-|\overline{\nabla}f|^{2}\leq-\overline{R}\leq 2m\;\Longrightarrow\:f\leq m\;\mbox{on $X$}.
\end{align*}
Therefore if $f\big|_{L}$ attains minimum on $L$, then $f\big|_{L} \equiv\ m$ is constant on $L$. Hence $\overline{R}\big|_{L} = -2m$, which means the minimum of $\overline{R}$ is attained on $L$. By {\cite{SZ}} Corollary 2.4, $\overline{g}$ is Einstein, a contradiction.

\end{proof}

\subsection{Infinitesimal Deformations of Lagrangian Translating Solitons}
We consider the special case that $F:L\looparrowright\mathbb{C}^{m}$ is a Lagrangian translating soliton. 
Let $(x_{1},\cdots,x_{m}, y_{1}, \cdots,y_{m})$ be standard coordinates in $\mathbb{R}^{2m}\simeq\mathbb{C}^{m}$. The Liouville form is defined by $\lambda := \sum_{i=1}^{m}y_{i}dx_{i}-x_{i}dy_{i}$. A Lagrangian submanifold $F: L\looparrowright\mathbb{C}^{m}$ is said to be {\it exact} if
\begin{align*}
F^{*}\lambda = d\beta
\end{align*}
for some $\beta\in C^{\infty}(L)$. The exact deformations (deformations that preserves exactness) are induced by exact $1$-forms on $L$, that is, $\xi\in\Gamma_{c}(NL)$ is an exact deformation if and only if $\widetilde{\omega}(\xi)$ is exact (see {\cite{NL}, Lemma 5.4).

We will restrict our attention to the study of exact deformations of an exact Lagrangian translating soliton $F:L\looparrowright\mathbb{C}^{m}$. In this case, $(\ref{fhar})$ reduces to the $f$-Laplace equation 
\begin{align}\label{Lap}
\Delta_{f}u = 0
\end{align}
on $L$, where $\Delta_{f} = -d^{*}_{f}d = \Delta - \frac{1}{2}\langle\nabla f, \:\cdot\:\rangle$ is the Witten Laplacian acting on functions. The solutions $u$ to $(\ref{Lap})$ are called $f$-harmonic functions on $L$. We show that there is a weighted $L^{2}$ gap between $f$-harmonic functions.

\begin{proposition}\label{L2}
Suppose $F:L\looparrowright\mathbb{C}^{m}$ is a Lagrangian translating soliton with $H+T^{\perp} = 0$, $|T| = 1$. Let $f(z) = 2\langle z, T\rangle$ and $u$ be a function on $L$ satisfying $\int_{L}u^{2} e^{-\frac{f}{2}}dV_{g} < \infty$ and $\Delta_{f}u = 0$. Then $u \equiv 0$.
\end{proposition}
\begin{proof}
Suppose $\Delta_{f}u = 0$. Let $w:= u^{2}$, then 
\begin{align*}
\Delta_{f}w = 2u\Delta u + 2|\nabla u|^{2} - u\langle\nabla u, \nabla f\rangle = 2|\nabla u|^{2}.
\end{align*}
Fix $p\in L$, consider a sequence of cut-off functions $\{\phi_{R}\}_{R>0}$ satisfying
\begin{align*}
\phi_{R} = 1\mbox{ in $B_{R}(p)$}, \;\;\phi_{R} = 0\mbox{ outside $B_{2R}(p)$},\;\mbox{ and }\;|\nabla\phi_{R}|\leq\frac{C}{R}
\end{align*}
for some $C>0$. Then
\begin{align*}
\begin{array}{rcl}
0 &=& \displaystyle\int_{L} div_{L}(e^{-\frac{f}{2}}\phi^{2}_{R}\nabla u^{2})\:dV_{g}\vspace{0.25cm}\\
    &=&\displaystyle\int_{L}\langle\nabla(e^{-\frac{f}{2}}\phi^{2}_{R}), \nabla u^{2}\rangle \:dV_{g} + \int_{L}e^{-\frac{f}{2}}\phi^{2}_{R}\Delta u^{2} \:dV_{g}\vspace{0.25cm}\\
    &=&\displaystyle\int_{L}4\langle u\nabla\phi_{R}, \phi_{R}\nabla u\rangle \:e^{-\frac{f}{2}}dV_{g} + \int_{L}2|\nabla u|^{2} \phi^{2}_{R}\:e^{-\frac{f}{2}}dV_{g}.
\end{array}
\end{align*}
By Young's inequality with $\epsilon = 1/2$ we then have
\begin{align*}
\int_{L}2|\nabla u|^{2} \phi^{2}_{R}\:e^{-\frac{f}{2}}dV_{g} &= -\int_{L}4\langle u\nabla\phi_{R}, \phi_{R}\nabla u\rangle \:e^{-\frac{f}{2}}dV_{g}\\ &\leq\int_{L}4u^{2}|\nabla\phi_{R}|^{2} \:e^{-\frac{f}{2}}dV_{g} + \int_{L}\phi^{2}_{R}|\nabla u|^{2}\:e^{-\frac{f}{2}}dV_{g}.
\end{align*}
Thus letting $R\to\infty$, by finiteness of $\int_{L}u^{2}\:e^{-\frac{f}{2}}dV_{g}$ we obtain $\int_{L}|\nabla u|^{2}e^{-\frac{f}{2}}dV_{g} = 0$, hence $u$ is constant. 

To show $u=0$, it is enough to show that $L$ has infinite weighted volume. First notice that we have the identities
\begin{align*}
\Delta f = -2|H|^{2} \mbox{ and } |H|^{2} + \frac{1}{4}|\nabla f|^{2} = 1,
\end{align*}
so
\begin{align*}
\Delta_{f}\:e^{\frac{f}{4}} = \frac{1}{4}\:e^{\frac{f}{4}}\left(\frac{1}{4}|\nabla f|^{2}-2\right)\leq -\frac{1}{4}\:e^{\frac{f}{4}}.
\end{align*}
From this we deduce that $\lambda_{1}(\Delta_{f})\geq\frac{1}{4}$ (see, for example, Proposition 22.2 of {\cite{Li}}). Then by a simple argument in {\cite{Vieira} Corollary 4.2}, we conclude that $L$ has infinite weighted volume. 
\end{proof}

From {\cite{NT}} proposition 2.2, we have $\Delta_{f}\theta = 0$ on Lagrangian translating solitons. Thus $\theta$ is $f$-harmonic. This corresponds to the fact that the exact deformation vector field induced by $\theta$ is just the mean curvature $H = J\nabla\theta$, which is just a translation in $\mathbb{C}^{m}$.

\begin{corollary}\label{CR}
If $F:L\looparrowright\mathbb{C}^{m}$ is a Lagrangian translating soliton as in Proposition $\ref{L2}$, $u\in C^{\infty}(L)$ with $\Delta_{f}u = 0$ and $\int_{L}(u-\theta)^{2}\:e^{-\frac{f}{2}}dV_{g}<\infty$, then $u\equiv\theta$.
\end{corollary}
\noindent Therefore there is no nontrivial exact deformation of exact Lagrangian translating solitons whose potential $u$ has finite weighted $L^{2}$ distance to the Lagrangian angle $\theta$. This provides a kind of infinitesimal uniqueness of exact Lagrangian translating solitons.

From $f$-harmonicity of $\theta$, we also have a nonexistence theorem in 2-dimensions.
\begin{corollary}
If $F:L\looparrowright\mathbb{C}^{m}$ is a Lagrangian translating surface as in Proposition $\ref{L2}$ with Lagrangian angle $\theta$ satisfying $\int_{L}\theta^{2}\:e^{-\frac{f}{2}}dV_{g}<\infty$, then $L$ is a plane. 
\end{corollary}
\begin{proof}
By proposition $\ref{L2}$, $\theta\equiv 0$ on $L$, so $T^{\perp} = -H = 0$, that is, $T$ is tangent to $L$. Therefore $L\simeq\Sigma\times\mathbb{R}\subset\mathbb{C}\times\mathbb{C}$ for some minimal curve $\Sigma$. Hence $\Sigma$ is a line and $L$ is a plane.
\end{proof}

\section{Generalizations and Related Problems}
\subsection{Generalization to Almost-Einstein Case}
Suppose now $(X, J, \:\overline{\omega})$ is a K\"ahler manifold and $f$ is a smooth function which is not necessarily a holomorphy potential. Then by the same computations as the proof of Theorem $\ref{SVF}$, one can show that the second variation formula of $f$-volume becomes
\begin{align}
(\delta^{2}V_{f})_{F(L)}(\xi) = \int_{L}\{\:\langle\:\widetilde{\omega}^{-1}\circ\Delta_{f}\circ\widetilde{\omega}(\xi),\: \xi\:\rangle - [\:\overline{\rho}(\xi, J\xi) + i\:\partial\overline{\partial}f(\xi,J\xi)\:]\:\}\;e^{-\frac{f}{2}}dV_{g},
\end{align}
where $\overline{\rho} = \overline{\Ric}(J\cdot, \cdot)$ is the Ricci form of $\overline{\omega}$.
In this case, the $f$-stability depends on the bilinear form $\overline{\rho}(\cdot, J\cdot) + i\:\partial\overline{\partial}f(\cdot,J\cdot)$. In particular, if $(X, J,\:\overline{\omega}, f)$ is \textit{almost-Einstein}, that is, 
\begin{align}
\overline{\rho} + i\:\partial\overline{\partial}f = C\:\overline{\omega}
\end{align}
for some $C\in\mathbb{R}$, then $\overline{\rho}(\cdot, J\cdot) + i\:\partial\overline{\partial}f(\cdot,J\cdot) = C\:\overline{g}(\cdot, \cdot)$. Thus
\begin{corollary}
Suppose $(X, J,\:\overline{\omega}, f)$ is almost-Einstein with $\overline{\rho} + i\:\partial\overline{\partial}f = C\:\overline{\omega}$ for some $C\in\mathbb{R}$. Then
\begin{enumerate}
\item[(i)] every $f$-minimal Lagrangian submanifold is $f$-stable if $C\leq 0$, and
\item[(ii)] if $C>0$, any compact $f$-minimal Lagrangian submanifold is Hamiltonian $f$-stable if and only if $\lambda_{1}(\Delta_{f})\geq C$.
\end{enumerate}
\end{corollary}
\noindent Notice that the above Hamiltonian $f$-stability criterion is also obtained in {\cite{KK}}.

\subsection{Generalized Lagrangian Mean Curvature Flow and Dynamic Stability}
A longstanding problem in Geometry is the existence problem for SLags in Calabi--Yau manifolds. Since SLags are volume minimizing, one approach to tackle the existence problem is to deform an initial Lagrangian submanifold along the negative gradient flow of the volume functional, namely, the mean curvature flow (MCF)
\begin{align*}
\left(\frac{d}{dt}F_{t}\right)^{\perp} = H(F_{t}).
\end{align*}
Smoczyk {\cite{Smo}} proves that the Lagrangian condition is preserved by MCF if the ambient space is K\"ahler-Einstein, and in this case the flow is called Lagrangian mean curvature flow (LMCF). However, finite-time singularities often occur and therefore in general one cannot have long-time existence and convergence. There are conjectural pictures in dealing with this problem, see for example, Thomas-Yau {\cite{TY}} and Joyce {\cite{JConj}}. 

A relevant question about long-time existence and convergence of LMCF one can ask is the relation between  stability of minimal Lagrangians under volume functional and \textit{dynamic stability} of LMCF, that is, whether a small Lagrangian perturbation of a stable minimal Lagrangian submanifold converges back to the original minimal submanifold along LMCF? Results in this direction can be found in, for instance, Li {\cite{HZLi}}, Tsai-Wang {\cite{TW1}}, {\cite{TW2}}, see also Lotay-Schulze {\cite{LS}} for an applications of {\cite{TW2}} to LMCF with singularities. 

The above picture can be generalized to $f$-minimal Lagrangians. More precisely, we consider the negative gradient flow of the $f$-volume functional:
\begin{align}\label{5}
\frac{d}{dt}F_{t} = H(F_{t}) + \frac{1}{2}(\overline{\nabla}f)^{\perp_{\overline{g},\:F_{t}}}.
\end{align}
Behrndt {\cite{TB}} shows that if $(X, J, \:\overline{\omega}, f)$ is almost-Einstein, then the Lagrangian condition is preserved by the flow $(\ref{5})$, called the \textit{generalized Lagrangian mean curvature flow (GLMCF)}. The stationary points of $(\ref{5})$ are the $f$-minimal Lagrangians. Therefore we can ask the same question about dynamic stability of GLMCF. Kajigaya--Kunikawa {\cite{KK}} recently generalized Li's result {\cite{HZLi}} and obtained a dynamic stability theorem for compact $f$-minimal Lagrangians in compact almost-Einstein K\"ahler manifolds. Besides the compact cases, the dynamic stability for LMCF solitons under GLMCF are especially interesting since in this case the GLMCF corresponds to LMCF with scalings.
\begin{problem}
Are the expanding and translating solitons for LMCF dynamically stable under GLMCF?
\end{problem}
\noindent The author believe that this problem is related to the conjectural theory of formation and desingularization of singularities of LMCF proposed by Joyce {\cite{JConj}}.

\subsection{K\"ahler--Ricci Mean Curvature Flow}\label{KRMCF}
There is another generalization of LMCF by considering the mean curvature flow along a moving ambient metric. Let $\{\overline{g}(t)\}$ be a solution to KRF, that is, its K\"ahler form $\overline{\omega}(t)$ satisfies
\begin{align}\label{KRF}
\frac{d}{dt}\overline{\omega}(t) = -\rho(\overline{\omega}(t)),\quad\mbox{ for $t\in(a, b)$},
\end{align}
where $\rho$ denotes the Ricci form. We consider the mean curvature flow $\{F_{t}\}$ along $\{\overline{g}(t)\}$, that is,
\begin{align}\label{MCF'}
\left(\frac{d}{dt}F_{t}\right)^{\perp_{\overline{g}(t), F_{t}}} = H(F_{t})\quad\mbox{ for $t\in(a, b)$},
\end{align}
where the mean curvature $H(F_{t})$ of $F_{t}$ is computed with respect to $\overline{g}(t)$. The couple $(\overline{g}(t), F_{t})$ defined by $(\ref{KRF})$ and $(\ref{MCF'})$ is called the K\"ahler--Ricci mean curvature flow (KR-MCF for short). Smoczyk {\cite{Smo}} shows that Lagrangian condition is preserved by KR--MCF.

Now, if we are given a gradient KR soliton $(X, J,\:\overline{\omega}, f)$, then there is a canonical KRF solution 
\begin{align}
\overline{g}(t) := \sigma(t)\:\varphi^{*}_{t}\overline{g},
\end{align}
which is defined for all $t$ such that $\sigma(t)>0$ (see Example $\ref{ExKRS}$), where $\varphi_{t}$ is the biholomorphism on $X$ generated by $\frac{1}{2\sigma(t)}\overline{\nabla}f$. In this case there is an one-to-one correspondence between GLMCF and KR-MCF, as shown in the following lemma.

\begin{lemma}\label{CF}
Let $(X, J, \:\overline{\omega}, f)$ be a gradient KR soliton and let $\overline{g}(t)$ be the solution to KRF defined as above. If $(\overline{g}(t), C_{t})$ is the solution to KR-MCF for $t\in(a, b)$, we set $F_{s(t)}:=\varphi_{t}\circ C_{t}$ for $s(t) = \int_{a}^{t}\frac{d\tau}{\sigma(\tau)}$. Then $F_{s}:L\looparrowright X$ satisfies the generalized LMCF in the fixed background $(X, J, \:\overline{\omega})$. Conversely, given a generalized LMCF $\{F_{s}\}$ in $(X, J,\:\overline{\omega})$, then $(\overline{g}(t), C_{t}:=(\varphi_{t})^{-1}\circ F_{s(t)})$ satisfies the KR-MCF.
\end{lemma}
\begin{proof}
Let $s(t)$ to be determined. We compute
\begin{align*}
\frac{d}{ds}F_{s} = \frac{dt}{ds}\:\frac{\partial}{\partial t}(\varphi_{t}\circ C_{t}) = \frac{dt}{ds}\left[\frac{1}{2\sigma(t)}(\overline{\nabla}f\circ\varphi_{t}\circ C_{t}) + d\varphi_{t}(H(C_{t}))\right].
\end{align*}
By solving $\frac{dt}{ds} = \sigma(t)$, we obtain $s(t) = \int_{a}^{t}\frac{d\tau}{\sigma(\tau)}$. Then by taking the normal component with respect to $\overline{g}$ and $F_{s}$,
\begin{align*}
\left(\frac{d}{ds}F_{s}\right)^{\perp_{\overline{g},\:F_{s}}} &= \frac{1}{2}(\overline{\nabla}f)^{\perp_{\overline{g},\:F_{s}}} + \sigma(t)\:d\varphi_{t}(H(C_{t}))\\
&= \frac{1}{2}(\overline{\nabla}f)^{\perp_{\overline{g},\:F_{s}}} + H(F_{s}).
\end{align*}
The converse follows from similar calculations.
\end{proof}
\noindent Notice that the case for shrinking solitons in shrinking Ricci solitons has been proved by Yamamoto {\cite{Yamamoto1}}, {\cite{Yamamoto2}}.

If we put $F_{s} = F:L\looparrowright X$ to be an $f$-minimal Lagrangian, then the KR-MCF evolves $F$ by $C_{t} = (\varphi_{t})^{-1}\circ F$, defined for all $t$ such that $\sigma(t) = 1-ct>0$. Therefore we conclude that

\begin{proposition}
An $f$-minimal Lagrangian submanifold $F:L\looparrowright X$ in a gradient KR soliton $(X, J, \:\overline{\omega}, f)$ generates a solution $C_{t} = (\varphi_{t})^{-1}\circ F$ to KR-MCF. Moreover, the solution $C_{t}$ is
\begin{enumerate}
\item[(i)] ancient if $(X, J, \:\overline{\omega}, f)$ is a shrinking soliton,
\item[(ii)] immortal if $(X, J, \:\overline{\omega}, f)$ is an expanding soliton, and
\item[(iii)] eternal if $(X, J, \:\overline{\omega}, f)$ is a steady soliton.
\end{enumerate}
\end{proposition}
Yamamoto {\cite{Yamamoto1}}, {\cite{Yamamoto2}} shows that if the Ricci flow and the Ricci-mean curvature flow develop type-I singularities at the same point simultaneously, then the blow-up near the singular point is an $f$-minimal submanifold in a shrinking Ricci soliton. It would be interesting to see how other $f$-minimal submanifolds arise as local models for the singularities (in particular, type-II singularities) of KR-MCF.

\bibliographystyle{alpha}
\bibliography{fminlag.bib}

\begin{bibdiv}
\begin{biblist}

\bib{ALW}{article}{
      author={Andrews, Ben},
      author={Li, Haizhong},
      author={Wei, Yong},
       title={{$\scr F$}-stability for self-shrinking solutions to mean
  curvature flow},
        date={2014},
        ISSN={1093-6106},
     journal={Asian J. Math.},
      volume={18},
      number={5},
       pages={757\ndash 777},
         url={http://dx.doi.org/10.4310/AJM.2014.v18.n5.a1},
      review={\MR{3287002}},
}

\bib{AS}{article}{
      author={Arezzo, Claudio},
      author={Sun, Jun},
       title={Self-shrinkers for the mean curvature flow in arbitrary
  codimension},
        date={2013},
        ISSN={0025-5874},
     journal={Math. Z.},
      volume={274},
      number={3-4},
       pages={993\ndash 1027},
         url={http://dx.doi.org/10.1007/s00209-012-1104-y},
      review={\MR{3078255}},
}

\bib{TB}{incollection}{
      author={Behrndt, Tapio},
       title={Generalized {L}agrangian mean curvature flow in {K}\"ahler
  manifolds that are almost {E}instein},
        date={2011},
   booktitle={Complex and differential geometry},
      series={Springer Proc. Math.},
      volume={8},
   publisher={Springer, Heidelberg},
       pages={65\ndash 79},
         url={https://doi.org/10.1007/978-3-642-20300-8_3},
      review={\MR{2964468}},
}

\bib{RB}{article}{
      author={Bryant, Robert~L.},
       title={Gradient {K}\"ahler {R}icci solitons},
        date={2008},
        ISSN={0303-1179},
     journal={Ast\'erisque},
      number={321},
       pages={51\ndash 97},
        note={G\'eom\'etrie diff\'erentielle, physique math\'ematique,
  math\'ematiques et soci\'et\'e. I},
      review={\MR{2521644}},
}

\bib{CH}{article}{
      author={Cao, Huai-Dong},
      author={Hamilton, Richard~S.},
       title={Gradient {K}\"ahler-{R}icci solitons and periodic orbits},
        date={2000},
        ISSN={1019-8385},
     journal={Comm. Anal. Geom.},
      volume={8},
      number={3},
       pages={517\ndash 529},
         url={https://doi.org/10.4310/CAG.2000.v8.n3.a3},
      review={\MR{1775136}},
}

\bib{BYB}{book}{
      author={Chen, Bang-Yen},
       title={Geometry of submanifolds and its applications},
   publisher={Science University of Tokyo, Tokyo},
        date={1981},
      review={\MR{627323}},
}

\bib{CM}{article}{
      author={Colding, Tobias~H.},
      author={Minicozzi, William~P., II},
       title={Generic mean curvature flow {I}: generic singularities},
        date={2012},
        ISSN={0003-486X},
     journal={Ann. of Math. (2)},
      volume={175},
      number={2},
       pages={755\ndash 833},
         url={http://dx.doi.org/10.4007/annals.2012.175.2.7},
      review={\MR{2993752}},
}

\bib{Da}{article}{
      author={Dazord, Pierre},
       title={Sur la g\'eom\'etrie des sous-fibr\'es et des feuilletages
  lagrangiens},
        date={1981},
        ISSN={0012-9593},
     journal={Ann. Sci. \'Ecole Norm. Sup. (4)},
      volume={14},
      number={4},
       pages={465\ndash 480 (1982)},
         url={http://www.numdam.org/item?id=ASENS_1981_4_14_4_465_0},
      review={\MR{654208}},
}

\bib{Jb3}{book}{
      author={Gross, M.},
      author={Huybrechts, D.},
      author={Joyce, D.},
       title={Calabi-{Y}au manifolds and related geometries},
      series={Universitext},
   publisher={Springer-Verlag, Berlin},
        date={2003},
        ISBN={3-540-44059-3},
         url={https://doi.org/10.1007/978-3-642-19004-9},
        note={Lectures from the Summer School held in Nordfjordeid, June 2001},
      review={\MR{1963559}},
}

\bib{HL}{article}{
      author={Harvey, Reese},
      author={Lawson, H.~Blaine, Jr.},
       title={Calibrated geometries},
        date={1982},
        ISSN={0001-5962},
     journal={Acta Math.},
      volume={148},
       pages={47\ndash 157},
         url={https://doi.org/10.1007/BF02392726},
      review={\MR{666108}},
}

\bib{IR2}{article}{
      author={Impera, Debora},
      author={Rimoldi, Michele},
       title={Rigidity results and topology at infinity of translating solitons
  of the mean curvature flow},
        date={2014},
      eprint={arXiv:1410.1139},
}

\bib{IR1}{article}{
      author={Impera, Debora},
      author={Rimoldi, Michele},
       title={Stability properties and topology at infinity of {$f$}-minimal
  hypersurfaces},
        date={2015},
        ISSN={0046-5755},
     journal={Geom. Dedicata},
      volume={178},
       pages={21\ndash 47},
         url={http://dx.doi.org/10.1007/s10711-014-9999-6},
      review={\MR{3397480}},
}

\bib{JConj}{article}{
      author={Joyce, Dominic},
       title={Conjectures on {B}ridgeland stability for {F}ukaya categories of
  {C}alabi-{Y}au manifolds, special {L}agrangians, and {L}agrangian mean
  curvature flow},
        date={2015},
        ISSN={2308-2151},
     journal={EMS Surv. Math. Sci.},
      volume={2},
      number={1},
       pages={1\ndash 62},
         url={https://doi.org/10.4171/EMSS/8},
      review={\MR{3354954}},
}

\bib{Jb2}{book}{
      author={Joyce, Dominic~D.},
       title={Riemannian holonomy groups and calibrated geometry},
      series={Oxford Graduate Texts in Mathematics},
   publisher={Oxford University Press, Oxford},
        date={2007},
      volume={12},
        ISBN={978-0-19-921559-1},
      review={\MR{2292510}},
}

\bib{KK}{article}{
      author={{Kajigaya}, T.},
      author={{Kunikawa}, K.},
       title={{Hamiltonian stability for weighted measure and generalized
  Lagrangian mean curvature flow}},
        date={2017-10},
     journal={ArXiv e-prints},
      eprint={1710.05537},
}

\bib{LL}{article}{
      author={Lee, Yng-Ing},
      author={Lue, Yang-Kai},
       title={The stability of self-shrinkers of mean curvature flow in higher
  co-dimension},
        date={2015},
        ISSN={0002-9947},
     journal={Trans. Amer. Math. Soc.},
      volume={367},
      number={4},
       pages={2411\ndash 2435},
         url={http://dx.doi.org/10.1090/S0002-9947-2014-05969-2},
      review={\MR{3301868}},
}

\bib{HZLi}{article}{
      author={Li, Haozhao},
       title={Convergence of {L}agrangian mean curvature flow in
  {K}\"ahler-{E}instein manifolds},
        date={2012},
        ISSN={0025-5874},
     journal={Math. Z.},
      volume={271},
      number={1-2},
       pages={313\ndash 342},
         url={https://doi.org/10.1007/s00209-011-0865-z},
      review={\MR{2917146}},
}

\bib{LZ}{article}{
      author={Li, Jiayu},
      author={Zhang, Yongbing},
       title={Lagrangian {F}-stability of closed {L}agrangian self-shrinkers},
        date={2017},
        ISSN={0075-4102},
     journal={J. Reine Angew. Math.},
      volume={733},
       pages={1\ndash 23},
         url={https://doi.org/10.1515/crelle-2015-0002},
      review={\MR{3731322}},
}

\bib{Li}{book}{
      author={Li, Peter},
       title={Geometric analysis},
      series={Cambridge Studies in Advanced Mathematics},
   publisher={Cambridge University Press, Cambridge},
        date={2012},
      volume={134},
        ISBN={978-1-107-02064-1},
         url={http://dx.doi.org/10.1017/CBO9781139105798},
      review={\MR{2962229}},
}

\bib{NL}{article}{
      author={Lotay, Jason~D.},
      author={Neves, Andr\'e},
       title={Uniqueness of {L}angrangian self-expanders},
        date={2013},
        ISSN={1465-3060},
     journal={Geom. Topol.},
      volume={17},
      number={5},
       pages={2689\ndash 2729},
         url={http://dx.doi.org/10.2140/gt.2013.17.2689},
      review={\MR{3190297}},
}

\bib{LS}{article}{
      author={Lotay, Jason~D},
      author={Schulze, Felix},
       title={Consequences of strong stability of minimal submanifolds},
        date={2018},
     journal={International Mathematics Research Notices},
       pages={rny095},
  eprint={/oup/backfile/content_public/journal/imrn/pap/10.1093_imrn_rny095/2/rny095.pdf},
         url={http://dx.doi.org/10.1093/imrn/rny095},
}

\bib{MC}{article}{
      author={Ma, Li},
      author={Chen, Dezhong},
       title={Remarks on complete non-compact gradient {R}icci expanding
  solitons},
        date={2010},
        ISSN={0386-5991},
     journal={Kodai Math. J.},
      volume={33},
      number={2},
       pages={173\ndash 181},
         url={https://doi.org/10.2996/kmj/1278076334},
      review={\MR{2681532}},
}

\bib{Mclean}{article}{
      author={McLean, Robert~C.},
       title={Deformations of calibrated submanifolds},
        date={1998},
        ISSN={1019-8385},
     journal={Comm. Anal. Geom.},
      volume={6},
      number={4},
       pages={705\ndash 747},
         url={http://dx.doi.org/10.4310/CAG.1998.v6.n4.a4},
      review={\MR{1664890}},
}

\bib{MW14}{article}{
      author={Munteanu, Ovidiu},
      author={Wang, Jiaping},
       title={Holomorphic functions on {K}\"ahler-{R}icci solitons},
        date={2014},
        ISSN={0024-6107},
     journal={J. Lond. Math. Soc. (2)},
      volume={89},
      number={3},
       pages={817\ndash 831},
         url={https://doi.org/10.1112/jlms/jdt078},
      review={\MR{3217651}},
}

\bib{MW15}{article}{
      author={Munteanu, Ovidiu},
      author={Wang, Jiaping},
       title={K\"ahler manifolds with real holomorphic vector fields},
        date={2015},
        ISSN={0025-5831},
     journal={Math. Ann.},
      volume={363},
      number={3-4},
       pages={893\ndash 911},
         url={https://doi.org/10.1007/s00208-015-1192-1},
      review={\MR{3412346}},
}

\bib{NT}{article}{
      author={Neves, Andr\'e},
      author={Tian, Gang},
       title={Translating solutions to {L}agrangian mean curvature flow},
        date={2013},
        ISSN={0002-9947},
     journal={Trans. Amer. Math. Soc.},
      volume={365},
      number={11},
       pages={5655\ndash 5680},
         url={http://dx.doi.org/10.1090/S0002-9947-2013-05649-8},
      review={\MR{3091260}},
}

\bib{Oh}{article}{
      author={Oh, Yong-Geun},
       title={Second variation and stabilities of minimal {L}agrangian
  submanifolds in {K}\"ahler manifolds},
        date={1990},
        ISSN={0020-9910},
     journal={Invent. Math.},
      volume={101},
      number={2},
       pages={501\ndash 519},
         url={http://dx.doi.org/10.1007/BF01231513},
      review={\MR{1062973}},
}

\bib{S}{article}{
      author={Shahriyari, Leili},
       title={Translating graphs by mean curvature flow},
        date={2015},
        ISSN={0046-5755},
     journal={Geom. Dedicata},
      volume={175},
       pages={57\ndash 64},
         url={http://dx.doi.org/10.1007/s10711-014-0028-6},
      review={\MR{3323629}},
}

\bib{Smo}{inproceedings}{
      author={{Smoczyk}, K.},
       title={{A canonical way to deform a Lagrangian submanifold}},
        date={1996-05},
   booktitle={eprint arxiv:dg-ga/9605005},
}

\bib{Sun}{article}{
      author={Sun, Jun},
       title={Lagrangian $l$-stability of lagrangian translating solitons},
        date={2016},
      eprint={arXiv:1612.06815},
}

\bib{TY}{article}{
      author={Thomas, R.~P.},
      author={Yau, S.-T.},
       title={Special {L}agrangians, stable bundles and mean curvature flow},
        date={2002},
        ISSN={1019-8385},
     journal={Comm. Anal. Geom.},
      volume={10},
      number={5},
       pages={1075\ndash 1113},
         url={https://doi.org/10.4310/CAG.2002.v10.n5.a8},
      review={\MR{1957663}},
}

\bib{TW2}{article}{
      author={{Tsai}, C.-J.},
      author={{Wang}, M.-T.},
       title={{A strong stability condition on minimal submanifolds and its
  implications}},
        date={2017-10},
     journal={ArXiv e-prints},
      eprint={1710.00433},
}

\bib{TW1}{article}{
      author={Tsai, Chung-Jun},
      author={Wang, Mu-Tao},
       title={Mean curvature flows in manifolds of special holonomy},
        date={2018},
        ISSN={0022-040X},
     journal={J. Differential Geom.},
      volume={108},
      number={3},
       pages={531\ndash 569},
         url={https://doi.org/10.4310/jdg/1519959625},
      review={\MR{3770850}},
}

\bib{Vieira}{article}{
      author={Vieira, Matheus},
       title={Harmonic forms on manifolds with non-negative
  {B}akry-\'emery-{R}icci curvature},
        date={2013},
        ISSN={0003-889X},
     journal={Arch. Math. (Basel)},
      volume={101},
      number={6},
       pages={581\ndash 590},
         url={http://dx.doi.org/10.1007/s00013-013-0594-0},
      review={\MR{3133732}},
}

\bib{Xin}{article}{
      author={Xin, Y.~L.},
       title={Translating solitons of the mean curvature flow},
        date={2015},
        ISSN={0944-2669},
     journal={Calc. Var. Partial Differential Equations},
      volume={54},
      number={2},
       pages={1995\ndash 2016},
         url={https://doi.org/10.1007/s00526-015-0853-y},
      review={\MR{3396441}},
}

\bib{Yamamoto1}{article}{
      author={{Yamamoto}, H.},
       title={{Ricci-mean curvature flows in gradient shrinking Ricci
  solitons}},
        date={2015-01},
     journal={ArXiv e-prints},
      eprint={1501.06256},
}

\bib{Yamamoto2}{article}{
      author={Yamamoto, Hikaru},
       title={Lagrangian self-similar solutions in gradient shrinking
  {K}\"ahler-{R}icci solitons},
        date={2017},
        ISSN={0047-2468},
     journal={J. Geom.},
      volume={108},
      number={1},
       pages={247\ndash 254},
         url={https://doi.org/10.1007/s00022-016-0336-0},
      review={\MR{3631881}},
}

\bib{Yang}{article}{
      author={Yang, Liuqing},
       title={Hamiltonian {L}-stability of {L}agrangian translating solitons},
        date={2015},
        ISSN={0046-5755},
     journal={Geom. Dedicata},
      volume={179},
       pages={169\ndash 176},
         url={http://dx.doi.org/10.1007/s10711-015-0073-9},
      review={\MR{3424663}},
}

\bib{SZ}{article}{
      author={Zhang, Shi~Jin},
       title={On a sharp volume estimate for gradient {R}icci solitons with
  scalar curvature bounded below},
        date={2011},
        ISSN={1439-8516},
     journal={Acta Math. Sin. (Engl. Ser.)},
      volume={27},
      number={5},
       pages={871\ndash 882},
      review={\MR{2786449}},
}

\end{biblist}
\end{bibdiv}

\

\noindent WEI-BO SU\vspace{0.2cm}\\
DEPARTMENT OF MATHEMATICS\\
NATIONAL TAIWAN UNIVERSITY\\
TAIPEI, TAIWAN\\
Email address: d03221004@ntu.edu.tw\\

\end{document}